\documentclass[12pt]{article}
\usepackage{amsmath,amsfonts,amssymb,amsthm,amscd}
\title{    \LARGE   Dependent   measures in independent theories\\ 
}

\author{Karim Khanaki\thanks{Partially supported by IPM grant 1400030118}\\Arak University of Technology}

\newtheorem{Theorem}{Theorem}[section]
\newtheorem{Proposition}[Theorem]{Proposition}
\newtheorem{Definition}[Theorem]{Definition}
\newtheorem{Remark}[Theorem]{Remark}
\newtheorem{Lemma}[Theorem]{Lemma}
\newtheorem{Corollary}[Theorem]{Corollary}
\newtheorem{Fact}[Theorem]{Fact}
\newtheorem{Example}[Theorem]{Example}
\newtheorem{Question}[Theorem]{Question}

\def\dotminus{\mathbin{\ooalign{\hss\raise1ex\hbox{.}\hss\cr
  \mathsurround=0pt$-$}}}

\begin{document}
\maketitle

\begin{abstract}   We introduce the notion of {\em dependence},  as a property of a Keisler measure, and generalize some results of \cite{HPS} in $NIP$ theories  to arbitrary theories. 	Among other things, we show that this notion is  very natural and fundamental for several reasons: (i)~all measures in $NIP$ theories are dependent, (ii) all types and all $fim$ measures in any theory are dependent, and (iii) as a  crucial result in measure theory, the  Glivenko-Cantelli class of functions (formulas) is characterized by dependent measures.
\end{abstract}

\section{Introduction} \label{1} 
The aim of this paper is to introduce and study a concept that we call dependent measure, which has deep roots in measure theory and the concept of $\mu$-stability within it. 
We will show that some of the results from \cite{HPS}, which were obtained under the assumption of $NIP$, hold for dependent measures in any arbitrary theory.
Among these results, we can mention the symmetry property of the Morley products of a measure and its approximation by average of types.

 It is worth noting that some of the arguments presented in this article are inherently similar to \cite{HPS}, although the key notion `dependent measure' allows us to use  facts in measure theory and make a connection between them and model theory.  {\em Surprisingly, in this case, model theory and measure theory have organic relationships, and the results of one domain corresponds to the results of another domain.}  We believe that the new notion of  dependent measure is valuable in itself and this approach can have more applications in future work.

The relationship between $NIP$ and analysis/topology was examined and highlighted in papers  \cite{Simon-Rosenthal}, 
\cite{Iba14}, and
\cite{K-amenable}.  In  \cite{Simon-Rosenthal}, the historical roots of this connection in analysis are presented, and some results of this analytical development have also been proven for $NIP$ formulas. The difference between the approach and results presented in the current paper and those in Simon’s paper lies in the fact that here we study only a single measure and do not assume the $NIP$  for the formula(s). In fact, if every measure possesses the dependence property (as defined in the present paper), we would arrive precisely at the results of Simon’s paper. This distinction makes the results of the current paper more measure-theoretic in nature, in contrast to Simon’s, which are topological.
This intrinsic similarity between Simon’s results and the findings of the present work becomes apparent in Fact~\ref{closure} below.

Let us give our motivation and point of view on the importance of this work.  It is very natural to   generalize the results in the $NIP$ context to arbitrary theories as this approach has already been pursued by generalizations of stable to simple and $NIP$ theories. 
On the other hand, this study will clarify why certain arguments work in $NIP$ theories and how they can be generalized to  outside this context.
Finally, this study identifies the deep links between two different areas of mathematics, namely model theory and measure theory, and their importance in applications as the results of \cite{HPS} are evidence of the usefulness of links between different domains, in the latter case probability theory and model theory.
Apart from these, studying `measures' as mathematical objects, which are the natural generalization of types, is interesting in itself and important in applications.
 
We have listed some of the most important results/observations to make it easier to go through the paper:
   Theorems~\ref{almost fim}, \ref{commute}, 
 \ref{CGH-like},  and Proposition~\ref{fim is dependent}, and  Corollary \ref{symmetry}.

This paper is organized as follows: In the next section we review some basic notions from measure theory. In Section 3  we introduce the notion {\em dependent measure} and give some basic properties of dependent measures. In  Section 4  we generalize some results of \cite{HPS} on the Morley products of measures and symmetric  measures to arbitrary theories. In Section 5
we study the relationship between the two concepts dependence and $fim$.
 We also prove a complement of a result of \cite{CGH} and give some new ideas for the future work. In ``Concluding remarks/questions'' we   will discuss the importance of this approach and questions will be raised about them.

\section{Preliminaries from measure theory} \label{measure theory}  
In this  section we give definitions from  measure theory
with which we be shall concerned, especially the notion of a stable set of functions (or $\mu$-stability) and its properties.

Let $X$ be a compact Hausdorff space. The space of continuous real-valued functions on $X$ is denoted here by $C(X)$.
 The smallest $\sigma$-algebra
containing the open sets is called the class of Borel sets.
By a Borel measure on $X$ we mean a finite measure defined
for all  Borel  sets. A Radon measure on $X$ is a Borel measure which is regular. 
Recall that a measure is {\em complete} if for any null measurable set $E$  every $F\subseteq E$ is measurable.
It is known that every Borel measure on a compact space has a unique extension to a complete measure. {\bf In this paper, we always assume that every Radon measure is complete.}

 In the following, given a measure $\mu$ and $k\geqslant 1$, the symbol  $\mu^k$ stands for $k$-fold product of $\mu$ and  $\mu^*$ stands for  the outer measure of  $\mu$. 
 
 \medskip
 The following fundamental concept introduced by David H.  Fremlin, namely the concept of $\mu$-stability, is a property for a set of functions relative to a fixed measure. In  model theory, we will  fix the set of functions/formulas and introduce its corresponding concept for measures, namely {\em dependent measure}.\footnote{We emphasis that Fremlin's use of the word ``stable" is not directly connected to the use of this word in model theory.} (Cf. Definition~3.2 for the definition of dependent measure.)
 From a logical point of view, this notion was first studied in \cite{K-amenable} in the framework of {\em  integral logic}.

\begin{Definition}[$\mu$-stability]  \label{Talagrand-stable}
	Let $A\subseteq C(X)$ be a pointwise bounded family
	of real-valued continuous functions  on $X$. Suppose that $\mu$
	is a Borel probability measure on $X$. We say that $A$ is {\em $\mu$-stable}, if   there is {\bf no}   measurable subset $E\subseteq X$ 	with  $\mu(E)>0$ and $s<r$  such that for each $k=\{1,\ldots k\}$ $$\mu^k\Big\{\overline{w}\in E^k: \ \forall I \subseteq k \ \exists f\in A \ \bigwedge_{i\in I}f(w_i)\leqslant s \wedge \bigwedge_{i\notin I} f(w_i)\geqslant r  \Big\}=(\mu E)^k.$$
\end{Definition}

\begin{Remark}  \label{remark-00}
	(i) The notion $\mu$-stable is an adaptation of \cite[465B]{Fremlin4}. Indeed,   by Proposition~4 of \cite{Talagrand}, it is easy to check the  equivalence. For this, notice that every function $f\in A$ is continuous on $X$ and so the left set in the equation above is measurable. This means that $(M)$ property of Proposition~4 of \cite{Talagrand} holds. 
	\newline
	(ii) A set $A$ of continuous functions on $X$ is stable with respect to $\mu$  iff $A$ is stable with respect to its completion $\bar\mu$. Indeed, recall that the product measures of $\mu,\bar\mu$ are the same. (See Proposition~465C(i) of \cite{Fremlin4}--Version of 26.8.13.)
\end{Remark}

The following are important results  connecting the notion of `stable' set of functions.\footnote{The article \cite{Simon-Rosenthal} by Simon is one that includes  results along the lines of Fact 2.3 and explains the connection between $NIP$ and a topological property of the closure of a set of functions/formulas.}
\begin{Fact}[\cite{Fremlin4}, Proposition~465D(b)]  \label{closure}
	Let $X$ be a compact Hausdorff space, $\mu$ a Radon probability measure on $X$, and $A\subseteq C(X)$. If $A$ is $\mu$-stable, then every function in the pointwise closure of $A$ is $\mu$-measurable.
\end{Fact}
As mentioned above, in this paper, we assume that all Radon measures are complete. Note that the completeness of $\mu$ is absolutely necessary in Fact~\ref{closure}.

Recall that the convex hull of $A\subseteq C(X)$, denoted by $\Gamma(A)$ or $\text{conv}(A)$, is the set of all convex combinations of functions in $A$, that is, the set of functions of the form $\sum_1^k r_i\cdot f_i$ where $k\in \Bbb N$, $f_i\in A$, $r_i\in{\Bbb R}^+$ and $\sum_1^k r_i=1$. 

The following theorem is the fundamental result on stable sets of functions. It  asserts that a set of functions  is stable iff it is a Glivenko--Cantelli class  iff  its convex  hull  is a Glivenko--Cantelli class (cf. \cite{K-GC}). 
\begin{Fact}[\cite{Fremlin4}] \label{fundamental}
Let $X$ be a compact Hausdorff space, $\mu$ a Radon probability measure on $X$, and $A\subseteq C(X)$ uniformly bounded. Then the following are equivalent:

(i) $A$ is $\mu$-stable.

(ii)  The convex  hull of $A$ is $\mu$-stable.

(iii)  $\lim_{k\to\infty}\sup_{f\in  A}|\frac{1}{k}\sum_1^k f(w_i)-\int f|=0$ for almost all $w\in X^{\Bbb N}$.

\medskip\noindent
(Here, $w=(w_1,w_2,\ldots)\in X^{\Bbb N}$  and the measure on $X^{\Bbb N}$ is the usual product measure.)
\end{Fact}

\noindent {\em Explanation.} The direction  (i)~$\implies$(ii) is  Theorem~465N(a) of \cite{Fremlin4}. The converse is evident. (See also Proposition~465C(a)(i) of \cite{Fremlin4}.)
The equivalence  (i)~$\iff$(iii) is the equivalence  (i)~$\iff$(ii)  of Theorem~465M of \cite{Fremlin4}. Again, we emphasize  that the completeness of $\mu$ is   necessary in the direction  (i)~$\implies$(iii).

\medskip
The last fact shows that, on stable sets  of functions, the topology of pointwise convergence is stronger than the topology of  convergence in measure.

\begin{Fact}[\cite{Fremlin4}, Thm.~465G] \label{convergence in measure}
Let $X$ be a compact Hausdorff space, $\mu$ a Radon probability measure on $X$, and $A\subseteq C(X)$ a $\mu$-stable set. Let $(f_i)$ be a net in $A$ such that $f_i\to f$ in the topology  of pointwise convergence. Then $\int|f_i-f|\to 0$.
\end{Fact}

\section{Dependent Keisler measures} \label{2} 
In this  section we introduce the notion of dependent measure (Definition~\ref{dependent measure}), and give some   of its principal  properties and examples.

The model theory notation is standard, and a text such as \cite{Simon} will be sufficient background.  We
fix a  first order language $L$, a complete   $L$-theory $T$ (not necessarily $NIP$), an $L$-formula $\phi(x,y)$, and a
subset $A$ of the monster model of $T$.  The monster model is denoted by $\cal U$. We let $\phi^*(y, x)
= \phi(x, y)$. 
 We define $p=tp_\phi(a/A)$ (where $a\in\cal U$ is a tuple of the  appropriate length) as the function
$\phi(p,y):A\to\{0,1\}$ by $b\mapsto\phi(a,b)$. This function is
called a complete $\phi$-type  over $A$. The set of all complete
$\phi$-types over $A$ is denoted by $S_\phi(A)$. We equip
$S_\phi(A)$ with the least topology in which all functions
$p\mapsto\phi(p,b)$ (for $b\in A$) are continuous. It is compact
and Hausdorff, and is  totally disconnected. Let
$X=S_{\phi^*}(A)$ be the space of complete
$\phi^*$-types over $A$.  Note that the functions
$q\mapsto\phi(a,q)$ (for $a\in A$) are {\em continuous}, and as
$\phi$ is fixed we can identify this set of functions with $A$.
So, $A$ is a subset of all bounded continuous functions on $X$,
denoted by $A\subseteq C(X)$.  

\medskip
A Keisler measure over $A$ in the variable $x$ is a {\em finitely}
additive probability measure on the Boolean algebra of $A$-definable  sets  in the variable $x$, denoted by $L_x(A)$. Every Keisler measure over $A$ can be represented by a regular Borel probability measure on  $S_x(A)$, the space of types  over $A$ in the variable $x$. A   measure   over $\cal U$ is called a {\em global} Keisler  measure.  The set of all measures over $A$ in the variable $x$ is denoted by ${\frak M}_x(A)$ or ${\frak M}(A)$. We will sometimes write $\mu$ as $\mu_x$ or $\mu(x)$ to emphasize that $\mu$ is a measure on the variable $x$.

For a formula $\phi(x,y)$, a Keisler $\phi$-measure over $A$ in the variable $x$ is a {\em finitely}
additive probability measure on the Boolean algebra of $\phi$-definable  sets over $A$ in the variable $x$, denoted by $L_\phi(A)$. Recall that a $\phi$-definable  set  over $A$ is a Boolean combination of the instances $\phi(x,b),b\in A$. The set of all $\phi$-measures over $A$ in the variable $x$ is denoted by  ${\frak M}_\phi(A)$.

\medskip
 Given a set $A$, an $L(A)$-formula $\theta(x)$, and types $p_1(x),\ldots,p_n(x)$ over $A$, the average measure of them (for $\theta(x)$), denoted by $\text{Av}(p_1,\ldots,p_n)$, is defined as follows:
$$\text{Av}(p_1,\ldots,p_n;\theta(x)):=\frac{\big|\{i:~\theta(x)\in p_i, i\leq n\}\big|}{n}.$$

 \medskip
We first revisit a useful dictionary (of well known facts which have been folklore for quite some time) 
 that is used in the rest of the paper. 
 In Fact~3.5 and Remark~3.6 of \cite{K-GC}, the proof of these facts, along with precise references to sources and their history, has been provided.
  For the definition of finitely satisfiable (definable, Borel definable) measures see Definition~7.16 of \cite{Simon}.

\begin{Fact} \label{Dirac} Let $T$ be a complete theory, $A$ a small set and $\phi(x,y)$ a formula.
	\newline
	(i)(Pillay \cite{Pillay}) There is a correspondence between global $A$-finitely satisfiable $\phi$-types $p(x)$ and the functions in the pointwise closure of   all functions $\phi(a,y):S_{\phi^*}(A)\to \{0,1\}$ for $a\in A$, where  $\phi(a,q)=1$ if and only if $\phi(a,y)\in q$.
	\newline
	(ii) The map $p\mapsto\delta_p$ is a correspondence   between global  $\phi$-types $p(x)$  and  Dirac measures $\delta_p(x)$ on $S_{\phi}({\cal U})$, where  $\delta_p(B)=1$ if $p\in B$, and $=0$ if   otherwise. Moreover,    $p(x)$ is finitely satisfiable in $A$ iff $\delta_p(x)$ is finitely satisfiable in $A$.
	\newline
	(iii) There is a correspondence between global  $\phi$-measures $\mu(x)$ and  regular Borel probability measures on $S_{\phi}({\cal U})$. Moreover, a  global  $\phi$-measure  is finitely satisfiable in $A$  iff its corresponding regular Borel probability measure is finitely satisfiable in $A$.
	\newline
	(iv) The closed convex hull of Dirac measures  $\delta(x)$  on $S_{\phi}({\cal U})$ is exactly all regular Borel probability measures $\mu(x)$ on $S_{\phi}({\cal U})$. Moreover,
	the closed convex hull of Dirac measures on $S_{\phi}({\cal U})$ which are finitely satisfiable in $A$   is exactly all regular Borel probability measures $\mu(x)$ on $S_{\phi}({\cal U})$ which are  finitely satisfiable in $A$. 
	\newline
	(v) There is a correspondence between global $A$-finitely satisfiable $\phi$-measures $\mu(x)$ and the functions in the pointwise closure of   all  functions of the form $\frac{1}{n}\sum_1^n \theta(a_i,y)$ on 
	$S_{\theta^*}(A)$, where $\theta$ is a $\phi^*$-formula,\footnote{Recall that a $\phi^*$-formula is a Boolean combination of instances of $\phi(a,y)$, $a\in  A$.} $a_i\in A$,  and  $\theta(a_i,q)=1$ if and only if $\theta(a_i,y)\in q$.
\end{Fact}
The above fact actually shows the ideology that we follow in this article. That is, the finitely satisfied (and invariant) measures in this paper are {\em functions} on specific topological spaces. The reader of this article can become more familiar with our approach as well as its historical trend by reading \cite{K3}, \cite{K-definable}, \cite{KP} and \cite{K-GC}.

\medskip\noindent
{\bf Convention.}  Recall that every regular Borel probability measure $\mu$ has a unique completion $\bar \mu$. In  this paper, without  loss of generality we can assume that every measure is complete. That is, $\mu=\bar\mu$. (The crucial notion of this paper (i.e. Definition~\ref{dependent measure}) is neutral to completion. Cf. Remark~\ref{remark-1}(ii) below.)

\medskip
 Let $\mu\in{\frak M}_x(A)$ and $\mu_\phi$ its restriction to $S_\phi(A)$ (equivalently, its restriction to the  Boolean algebra of $\phi$-formulas over $A$).
Notice that the restriction map $r_\phi:S_x(A)\to S_\phi(A)$ is a quotient map, and $\mu_\phi(X)=\mu(r_\phi^{-1}(X))$
 for any Borel subset $X\subseteq S_\phi(A)$. (See Remark~1.2 in \cite{CGH}.) For $A\subseteq B$ and $\mu\in {\frak M}_x(B)$, $\mu|_A\in {\frak M}_x(A)$ is the restriction of $\mu$ by the quotient map 
 $r:S_x(B)\to S_x(A)$. The restriction of $\mu|_A$ to $S_\phi(A)$ is denoted by $\mu_{\phi,A}$.
 
 \medskip
 The crucial notion of the paper is as follows. (Compare  Definition~\ref{Talagrand-stable}). First, we need to introduce a notation. Let $M$ be a model of $T$, $A\subseteq B\subseteq M$, and $\phi(x,y)$ an $L$-formula. For any $E\subseteq S_\phi(B)$  we write 
 \begin{align*}
 D_k(A, B, E, \phi) = \Big\{\overline{p}\in E^k: \ \forall I \subseteq k \ \exists b\in A \ \bigwedge_{i\in I}\phi(p_i,b)= 0 \wedge \bigwedge_{i\notin I} \phi(p_i,b)=1\Big\}.
 \end{align*}
 	\noindent
 (Recall that $\phi(p_i,b)= 1$ if $\phi(x,b)\in p_i$  and $\phi(p_i,b)=0$ in otherwise.)
\begin{Definition}[Dependent Measures]  \label{dependent measure}
	Let $T$ be a complete theory, $M$  a  set/model, and $\mu_x\in{\frak M}(M)$. 
	\newline
	(i) Suppose that $A\subseteq B\subseteq M$. We say that $\mu$ is {\em dependent over  $B$ and in $A$}, if for any formula  $\phi(x,y)$ there is {\bf no} $E\subseteq S_\phi(B)$  
	measurable, $\mu_{\phi,B}(E)>0$ (where $\mu_{\phi,B}=(\mu|_B)_\phi$)   such that for each $k$, $(\mu_{\phi,B}^{k})D_k(A, B, E, \phi)=(\mu_{\phi,B} E)^{k}$.
	\newline
	(ii) We say that $\mu$ is {\em dependent}, if   $\mu$ is dependent over $M$ and in $M$.  
\end{Definition}
The above parameters $A$ and $B$  have complicated the definition. The reason for using  parameters   is that  it gives us the possibility to expand or restrict functions/formulas and spaces, which will be used in some places. (In Remark~\ref{remark-1}(iii) below, we study monotonicity  properties of ``over $B$ and in $A$" as one varies $B$ and $A$.)

On the other hand, it is possible to provide a definition that can be explained by model-theoretic intuition. In fact, using the argument of Theorem~\ref{almost fim} below (or Fact~\ref{fundamental}), it is easy to show that $\mu_x\in{\frak M}(M)$ is dependent if and only if for any formula $\phi(x,y)$,
$$\mu^\omega\Big(\Big\{\bar p\in\prod_{i<\omega}S_{x_i}(M):  \lim_{k\to\infty}\sup_{b\in M}\big|\mu(\phi(x,b))-\frac{1}{k}\sum_{i=1}^{k} \phi(p_i,b)\big|=0  \Big\}\Big)=1.  \ \ \ \   (\Diamond)$$  
\newline
(Here $\mu^\omega$ is the usual product measure.)

\begin{Remark}  \label{remark-1}
	(i) The notion of    dependent measure is an adaptation of Definition~\ref{Talagrand-stable} above to the model theory context. Indeed,   note that every function $\phi(x,b)$ is continuous on $S_\phi(B)$ and so $D_k(A, B, E, \phi)$ is $\mu_{\phi,B}^k$-measurable:
	Although  $A$ is a potentially uncountable set, but   $D_k(A, B, E, \phi)$ is a union of the sets of the form $$E^k\cap\Big(\bigcap_{I\subset k}\{\bar p\in (S_\phi(B))^k: b_I\in A, \bigwedge_{i\in I}\phi(p_i,b_I)=0\wedge\bigwedge_{i\notin I}\phi(p_i,b_I)=1\}\Big)$$
	and the  set inside the parentheses is clopen. Now, it is easy to see that $D_k(A, B, E, \phi)$ is of the form $E^k\cap O$ where $O$ is an open set.
	 To summarize, $(\mu_{\phi,B}^k)^*(D_k(A, B, E, \phi))=\mu_{\phi,B}^k(D_k(A, B, E, \phi))$ and $(M)$\footnote{Abbreviation for measurability.} property of Proposition~4 of \cite{Talagrand} holds. 
	\newline
	(ii) Recall from Remark~\ref{remark-00}(ii) that a Keisler measure $\mu$ is dependent iff its completion $\bar\mu$ is dependent.   
	\newline
	(iii) (1) Let $A\subseteq B$. As the restriction map $r:S_x(B)\to S_x(A)$ is a quotient map, it is easy to verify that  $\mu|_A(X)=\mu|_B(r^{-1}(X))$
	for any Borel subset $X\subseteq S_x(A)$. (See Remark~1.2 in \cite{CGH}.) Now, by Proposition~465C(d) of \cite{Fremlin4}--Version of 26.8.13, if $\mu$ is dependent over $A$ and in $A$, then it is dependent over $B$ and in $A$. The converse is clear.
	(2) Let $A\subseteq B\subseteq M$. It is easy to verify that: if $\mu$ is dependent over $M$ and in $B$, then it is dependent over $M$ and in $A$. (Cf. \cite{Fremlin4}--Version of 26.8.13, Proposition~465C(a)(i).) By (1) and (2), $\mu\in{\frak M}(M)$ is dependent if and only if for any $A\subseteq B\subseteq M$, $\mu$ is dependent over $B$ and in $A$.
	\newline
	(iv) As the restriction map $r_\phi:S_x(B)\to S_\phi(B)$ is a quotient map, it is easy to verify that  $\mu_\phi(E)=\mu(r_\phi^{-1}(E))$
	for any Borel subset $E\subseteq S_\phi(B)$. (See Remark~1.2 in \cite{CGH}.) Therefore, by  Proposition~465C(d) of \cite{Fremlin4} again, if $\mu$ is dependent over $B$ and in $A$, then whenever $E\subseteq S_x(B)$ is
	measurable, $\mu|_B(E)>0$, there is some
	$k\geqslant 1$ such that $(\mu|_B)^{k}D_k(A,  B,E)<(\mu|_B E)^{k}$
	where $D_k(A,  B,E)=\big\{\overline{p}\in E^k: \ \forall I \subseteq k \ \exists b\in A \ \bigwedge_{i\in I}\phi(p_i,b)= 0 \wedge \bigwedge_{i\notin I} \phi(p_i,b)=1\big\}$.
\end{Remark}

The following is an important property of the notion dependent measure.
\begin{Proposition} \label{measurable-1}
	Let $T$ be a complete theory, $B$  a set, and $\mu\in{\frak M}_x(B)$. If $\mu$ is dependent, then for   any $A\subseteq  B$ and any $L(A)$-formula $\phi(x,y)$, every function in the pointwise closure of the convex hull of $\{\phi(x,b):S_\phi(A)\to\{0,1\}~| b\in A\}$ is  $\mu$-measurable. 
\end{Proposition}
\begin{proof} This is a  consequence of 	Facts~\ref{closure} and \ref{fundamental}. Indeed, recall that $\mu$ is 	dependent iff its completion $\bar\mu$ is. Now, by the above convention,   $\mu_{\phi,A}$ (is complete and) satisfies in the assumptions of Facts~\ref{closure} and \ref{fundamental}. 										
\end{proof}

The concept of $NIP$ (non independence property)  for formulas/functions was introduced by Shelah in the 1970s, which was also independently introduced and studied by Vapnik and Chervonenkis in learning theory. (Cf. \cite{Simon}, for the  definition of $NIP$ for a theory/formula.) In Section~5, a restricted version of this concept is studied. On the other hand, a weakened version of NIP for models/sets has recently been introduced by Khanaki and Pillay, namely `NIP in a model.'
Recall from \cite[Def.~1.1]{KP} that a formula $\phi(x,y)$ has  $NIP$ in a model $M$ if   there is {\bf no} countably infinite sequence $(a_i)\in M$ such that for   all finite disjoint subsets $I, J\subseteq \Bbb N$, $M\models \exists y(\bigwedge_{i\in I}\phi(a_i,y)\wedge\bigwedge_{i\in J}\neg\phi(a_i,y))$.
The formula $\phi$ has $NIP$ for the theory $T$ iff it has $NIP$ in every model $M$ of $T$ iff it has $NIP$ in some model $M$ of $T$ in which all types over the empty set in countably many variables are realised.

\begin{Proposition}  \label{type=dependent}
		(i) For any theory, every type is dependent.\footnote{The notion of a dependent type in Definition~1.5 of paper \cite{GOU} is different from the notion of dependence in the present paper, and the reader should be aware of and pay close attention to this distinction.}
	\newline 
	(ii) For any theory $T$ and any model $M$ of $T$, if every Keisler  measure over $M$  is dependent, then every formula   has  $NIP$ in    $M$.  
	\newline
	(iii) A  theory $T$ is  $NIP$ iff in any model $M$ of $T$, every Keisler  measure over $M$  is dependent iff every Keisler  measure over some model $M$ of $T$ in which all types over the empty set in countably many variables are realised is dependent.	
\end{Proposition}
\begin{proof} (i) is evident, by definition. (See also Theorem~5 of \cite{Talagrand}.)

	(ii): By Proposition~\ref{measurable-1}, for any model  $M$ and any formula $\phi(x,y)$, every function in the pointwise closure of $\{\phi(x,b):S_\phi(M)\to\{0,1\}~| b\in M\}$ is measurable with respect to  {\bf any} Keisler measure over $M$.
	By the equivalence (iv)~$\iff$~(vi) of Theorem~2F of \cite{BFT}, $\phi(x,y)$ has $NIP$ in $M$.
	
	(iii): It is easy to see that, for any formal $\phi(x,y)$, if $\phi(x,y)$ has $NIP$, then for some $k=k_\phi$ we have $D_k(A, B,E,\phi)=\emptyset$.  (Cf. Definition~\ref{dependent measure}.)   Conversely, if for any model $M$ of $T$, every measure over $M$ is dependent, then by (ii) every formula has $NIP$ in any model of $T$. By Remark~2.1 of \cite{KP}, this means that $\phi$ has $NIP$ for the theory $T$.
It can easily be verified that in the above equivalence, it is sufficient to consider some model $M$ of $T$ in which all types over the empty set in countably many variables are realised.
\end{proof}

\begin{Example} \label{example 0}
(i) By Proposition~\ref{type=dependent} above, all types in any theory, and all measures in $NIP$ theories are dependent.
\newline
(ii) We say that a measure $\mu$ is {\em purely atomic} if there are Dirac measures $(\delta_n:n<\omega)$ such that $\mu=\sum_{1}^{\infty}r_n.\delta_n$ where $r_n\in[0,1]$ and $\sum_{1}^{\infty}r_n=1$.  By definition, it is easy to verify that any   purely atomic  measure is dependent (cf. also Theorem~5 of \cite{Talagrand}). In \cite{CG}, a measure $\mu$ is called trivial, if (1) it is purely atomic (i.e. $\mu=\sum_{1}^{\infty}r_n.\delta_n$) , and (2) any $\delta_n$ is realized in $\cal U$, i.e. $\delta_n=tp(a_n/{\cal U})$ for some $a_n\in\cal U$.  It is shown  \cite[Theorem 4.9]{CG} that, in the theories of the random graph and the random bipartite graph, every definable and finitely satisfiable measure is trivial. This means that such measures are dependent. (Similarly, all definable and finitely satisfiable measures in the theories in \cite[Corollary 4.10]{CG} are dependent.)
\newline
(iii) Furthermore, in Proposition~\ref{fim is dependent} below, we show that any $fim$ measure (in any theory)  is dependent.
\end{Example}

 Recall from Proposition~\ref{type=dependent} that, in any non-$NIP$ theory, there is  a measure  that is not dependent.
In the following example we give a concrete example of a non-dependent measure. In Example~\ref{non-example} we will present another example.
\begin{Example} \label{non-example 2} We review the last example of \cite[Example~7.2]{Simon} and show  that this is a non-dependent measure. Let $T$ be the theory of the random graph in the language $\{R\}$. Let $M\models T$ be the unique countable model. Define $\mu$ on $M$ by $\mu\big(\bigcap_{i<n}(xRa_i)^{\eta(i)}\big)=2^{-n}$ for any choice of pairwise distinct $a_i$'s in $M$ and $\eta:n\to\{0,1\}$. It is not hard to see that $\mu$ is $\emptyset$-definable. We denote  the global definable extension of $\mu$ again by $\mu$.
 We claim that the global measure $\mu$ is not dependent. To check this, for each $k$,  we consider the following formulas:
$$\Phi_1:=\bigwedge_
{i\neq j, i,j\leq k}x_i\neq x_j,
$$
$$\Phi_2:=\forall I\subset k \ \exists y_I\big(\bigwedge_{i\in I}(x_i R y_I) \wedge\bigwedge_{i\notin I}\neg (x_i R y_I)\big).$$
Set $\Phi(x_1,\ldots,x_k):=\neg\Phi_1\vee\Phi_2$. Notice that, by the randomness, $\forall\bar x\Phi(\bar x)\equiv\top$ and so $\mu^{(k)}(\Phi)=1$. Also, an easy computation shows that $\mu^{(k)}(\Phi_1)=1$.  Therefore $\mu^{(k)}(\Phi_2)=1$. Set $D_k:=\{(p_1,\ldots,p_k)\in(S_{x_1}({\cal U}))^k: p|_{x_i}=p_i \  (i\leq k) \text{ for some } p\in S_{x_1,\ldots,x_k}({\cal U}) \text{ with } \Phi_2\in p\}$.
Notice that $D_k=D_k({\cal U},{\cal U},S_x({\cal U}))$ as in
Remark~\ref{remark-1}(iv). Also, similar to Remark~\ref{remark-1}(i)
one can check that $D_k$
is $\mu^k$-measurable.

On the other hand, by Proposition 3.3 of \cite{GH}, the restriction map $r:S_{x_1,\ldots,x_k}({\cal U})\to S_{x_1}({\cal U})^k$ via $p\mapsto (p_1,\ldots,p_k)$ where 
$p|_{x_i}=p_i$   ($i\leq k$), is a quotient map, and the {\em pushforward} of $\mu^{(k)}$ is $\mu^k$; that is, $\mu^{(k)}(r^{-1}(D))=\mu^k(D)$ for any $\mu^k$-measurable set $D\subseteq
S_{x_1}({\cal U})^k$.
 Therefore, $\mu^k(D_k)=
 \mu^{(k)}(r^{-1}(D_k))=
 \mu^{(k)}(\Phi_2)=1$.
 As $k$ is arbitrary, the measure $\mu$ is not dependent for the formula $xRy$.

Furthermore, as the only place where  the definition of $\mu$ is used is in the computation of $\mu^{(k)}(\Phi_1)$, the above argument  suggests that, in this theory, every non-trivial measure is non-dependent.

 Also, it is easy to check that $\mu$ is not finitely satisfiable in any small model. Alternatively, Conant and Gannon  \cite[Thm~4.9]{CG}	showed that every both definable and finitely satisfiable measure in $T$ is trivial, which the measure in this example clearly is not.
 On the other hand, it is shown in Fact 2.5 of \cite{GH}  that $\mu$ commutes with itself.
\end{Example}

\section{Dependence and symmetry}	
In this section we generalize some results of \cite{HPS} on the Morley products of measures and  symmetric measures.
The following is the fundamental property of the notion of dependent measure.
\begin{Theorem} \label{almost fim}
Let $T$ be a complete theory, $A$  a small set, and $\mu_x\in{\frak M}({\cal U})$. If $\mu|_A$ is dependent, then for any $\mu|_A$-measurable subsets $X_1,\ldots,X_m\subseteq S_x(A)$ and $\epsilon>0$, there are $n\in{\Bbb N}$ and $E\subseteq (S_x({\cal U}))^n$, with $(\mu^n)^* E\geq1-\epsilon$, such that for every $b\in A$ and $k\leq m$, $$\big|\mu(\phi(x,b)\cap X_k)-Av(p_1|_A,\ldots,p_n|_A; \phi(x,b)\cap X_k)\big|\leq\epsilon, \ \ \ (*)$$
for all $(p_1,\ldots,p_n)\in E$. (Here $p_i|_A$ is the restriction of $p_i$ to $S_x(A)$.)
\end{Theorem}
\begin{proof}
 First note that, by   Remark~\ref{remark-1}(iv), we can use $\mu|_A$ instead of $\mu_{\phi,A}$ (in Definition~\ref{dependent measure}). 
As $\mu|_A$ is dependent  (equivalently the set $\{\phi(x,b):S_x(A)\to \{0,1\}~|b\in A\}$ is $\mu|_A$-stable in the sense of Definition~\ref{Talagrand-stable}),  by Proposition~465C(a)(v) and (b)(ii) of \cite{Fremlin4}--Version of 26.8.13, the set $\bigcup_{k=1}^m\{\phi(x,b)\times \chi_{X_k}:b\in A\}$ is $\mu|_A$-stable, where $\chi_{X_k}$ is the characteristic function of $X_k$. Therefore,  by Fact~\ref{fundamental},  we have  $$\sup_{b\in  A}|\frac{1}{n}\sum_1^n\phi(p_i,b)\times\chi_{X_k}-\mu(\phi(x,b)\cap\chi_{X_k})|\to 0$$ as $n\to\infty$ for all $k\leq m$ and for almost every $(p_i)\in S_x(A)^{\Bbb N}$ with the product measure $(\mu|_A)^{\Bbb N}$.    Now, it is easy to verify that the claim holds. Indeed, one can  see directly (or using Theorem 11-1-1(c) of \cite{T84})  that there are $n\in{\Bbb N}$ and $F\subseteq (S_x(A))^n$, with $(\mu|_A ^n)^* F\geq1-\epsilon$, such that for every $b\in A$ and $k\leq m$,  $\big|\mu(\phi(x,b)\cap X_k)-Av(p_1',\ldots,p_n'; \phi(x,b)\cap X_k)\big|\leq\epsilon$ for all $(p_i')\in F$. 
Finally, use  Remark~\ref{remark-1}(iii) above  and find the desired set $E\subseteq (S_x({\cal U}))^n$ such that $(*)$ holds. (Here, $p_i'=p_i|_A$ for some $p_i\in S_x({\cal U})$.)
\end{proof}

In the following, for $\mu\in{\frak M}({\cal U})$,  the support of $\mu$ is denoted by $Supp(\mu)$. (Cf. \cite[p. 99]{Simon}.)
\begin{Corollary} \label{cor 1}
	Let $T$ be a complete theory, $A$  a small set, and $\mu\in{\frak M}_x({\cal U})$. If $\mu|_A$ is dependent, then for any $\mu|_A$-measurable subsets $X_1,\ldots,X_m\subseteq S_x(A)$ and $\epsilon>0$, there are  $p_1|_A,\ldots,p_n|_A\in S_x(A)$  such that for every $b\in A$ and $k\leq m$, $$\big|\mu(\phi(x,b)\cap X_k)-Av(p_1|_A,\ldots,p_n|_A; \phi(x,b)\cap X_k)\big|\leq\epsilon.$$
	Furthermore, we can assume that $p_i\in Supp(\mu)$ for all $i$.
\end{Corollary}
\begin{proof}
	Immediate, by Theorem~\ref{almost fim}. (Recall from  \cite[Proposition~2.10]{Gannon-thesis} that $\mu(Supp(\mu))=1$. This   assures us that we can assume that $p_i\in Supp(\mu)$ for all $i$.)
\end{proof}

The following result allows us to define the Morley product of a finitely satisfiable measure and a dependent measure.
\begin{Proposition} \label{measurable}
	Let $\mu_x$ be a global
	 $A$-finitely satisfied measure, $\lambda_y$ a global dependent measure and  $\phi(x,y;b)$ an $L({\cal U})$-formula.
	 Let $N\supseteq Ab$ be a model and define
the function $f:S_{\phi^*}(N)\to[0,1]$,\footnote{Recall that $\phi^*(y,x;b)=\phi(x,y;b)$.}  by $q\mapsto \mu(\phi(x,d;b))$ for some (any) $d\models q$. Then $f$ is $\lambda_y|_N$-measurable.
\end{Proposition}
\begin{proof}
	As $\mu$ is  $A$-finitely satisfied, by Fact~\ref{Dirac}(v), $f$ is in the closure of the convex hull of $\{\phi(a,y;b):S_{\phi^*}(N)\to\{0,1\}~|a\in A\}$.
 Now, as  $\lambda_y$ is dependent, by  Proposition~\ref{measurable-1}    above, $f$ is $\lambda_y|_N$-measurable.
  (Indeed, recall that $\lambda_y$ is a complete measure, with the above convention.)
\end{proof}

\begin{Definition} \label{Morley product}
 Under  the assumptions of Proposition~\ref{measurable} we define the Morley product measure $\mu(x)\otimes\lambda(y)$ as follows:
	$$\mu(x)\otimes\lambda(y)(\phi(x,y;b))=\int_{S_{\phi^*}(N)}f~ d\lambda|_N.$$
	It is easy to verify that the definition does not depend on the choice of $N$. We will sometimes write $f$ as $f_{\mu}^\phi$  (or $f_{\mu,N}^\phi$) to emphasize that it is related to $\mu$, $\phi$ (and $N$) as above. 
\end{Definition}

The following is a generalization of Lemma~7.1 of \cite{Simon}, although the proof is essentially the same, using the previous observations.
\begin{Lemma} \label{type-measure}
	Let $\mu_x,\lambda_y$ be  global dependent measures such that   $\mu_x$ is $A$-finitely satisfied (or Borel definable over $A$) and  $\lambda_y$ is  $A$-finitely satisfied. If $\mu_x\otimes  q_y=q_y\otimes\mu_x$ for any $q_y\in S_y({\cal U})$ in the support of $\lambda_y$, then $\mu_x\otimes\lambda_y=\lambda_y\otimes\mu_x$.
\end{Lemma}
\begin{proof}
	 Let $\phi(x,y;b)\in L({\cal U})$ and $N\supseteq Ab$ a model. Let $f=f_{\mu,N}^\phi$ be as above. As $\mu_x$ is $A$-finitely satisfied  and $\lambda_y$ is dependent (or just $\mu_x$ is Borel definable over $A$), the Morley product $\mu\otimes\lambda$ is well-defined. (Cf. Definition~\ref{Morley product} and Proposition~\ref{measurable}.)
	Fix $\epsilon>0$. Let $\sum_1^n r_i. \chi_{X_i}$ be a simple $\lambda|_N$-measurable function  such that $|f(q)-\sum_1^n r_i. \chi_{X_i}(q)|<\epsilon$ for all $q\in S_{\phi^*}(N)$. (That is,  $X_1,\ldots,X_n\in S_{\phi^*}(N)$ are $\lambda|_N$-measurable, $\chi_{X_i}$ is the characteristic function of $X_i$, and $r_i\in[0,1]$ for $i\leq n$.) By Corollary~\ref{cor 1}, there are $q_1,\ldots,q_m\in Supp(\lambda)$  such that if ${\tilde\lambda}=\frac{1}{m}\sum q_i$ then 
	
	(1) $|{\tilde\lambda}(X_i)-\lambda(X_i)|<\epsilon$ for all $i\leq n$, and 
	
	(2) $|{\tilde\lambda}(\phi(a,y;b))-\lambda(\phi(a,y;b))|<\epsilon$ for all $a\in\cal U$.
	
(Here, we let $r^{-1}(X_i):=X_i$ again, where  $r:S_{\phi^*}({\cal U})\to S_{\phi^*}(N)$ is the restriction map.)   Note that,  as the $q_i$'s are types, the product $\mu\otimes\tilde\lambda$ is well-defined.
We remained the reader  that the $q_i$'s are $A$-finitely satisfied because of the assumption that $\lambda$ is, and since they are in the support of $\lambda$.
  The product measure ${\tilde\lambda}\otimes\mu$ is well-defined since $\mu$ is dependent and  the $q_i$'s are $A$-finitely satisfied. As $\mu$ commutes with $\tilde\lambda$ and $\epsilon$ is arbitrary, by the conditions (1),(2), it is easy to see that  $\mu_x\otimes\lambda_y(\phi(x,y;b))=\lambda_y\otimes\mu_x(\phi(x,y;b))$.
\end{proof}
The above argument can be further visualized in the language of analysis. Given $\epsilon>0$, and $r,s\in\Bbb R$, we write $r \approx_\epsilon s$ to denote   $|r-s|<\epsilon$.
With the above assumptions, the argument of Lemma~\ref{type-measure} is as follows:
\begin{align*}
\mu\otimes\lambda(\phi(x,y;b)) & = \int f_\mu^\phi d\lambda \approx_\epsilon  \int(\sum r_i\cdot\chi_{X_i})d\lambda=\sum r_i\cdot\lambda(X_i)\\
&   \approx_\epsilon \sum r_i\cdot\tilde{\lambda}(X_i) = \int(\sum r_i\cdot\chi_{X_i})d\tilde{\lambda}  \approx_\epsilon \int f_\mu^\phi d\tilde{\lambda} \\
&   = \mu\otimes\tilde{\lambda}(\phi(x,y;b)) = \tilde{\lambda}\otimes\mu(\phi(x,y;b)) = \int f_{\tilde{\lambda}}^{\phi^*} d\mu  \\
&  \approx_\epsilon   \int f_{\lambda}^{\phi^*} d\mu =   \lambda\otimes\mu(\phi(x,y;b)).
\end{align*}

\begin{Remark}
Assuming that $\mu$ and  $\lambda$ are dependent, and using Lemma~\ref{type-measure}, one can give  a generalization of  \cite[Lemma~3.1]{HPS}. That is, if  $\mu\in{\frak M}_x({\cal U})$ is    definable over $A$ and dependent, and $\lambda\in{\frak M}_y({\cal U})$ is   $A$-finitely satisfied and \textbf{\textit{dependent}},  then   $\mu_x\otimes\lambda_y=\lambda_y\otimes\mu_x$. (See also \cite[Proposition~7.22]{Simon}.) Nevertheless, we  prove something stronger (cf. Theorem~\ref{commute} below). In fact, the dependence of $\lambda$ is \textbf{\textit{unnecessary}}. 
\end{Remark}

To proceed, we require a lemma that is independently significant and demonstrates the continuity of a map similar to Proposition~6.3 in \cite{ChG}, albeit constrained to the narrower domain of finitely satisfiable measures. Notably, establishing this result necessitates recalling that average measures form a dense subset in this space and leveraging the regularity of the space.
\begin{Lemma} \label{convergence 2}
	Let $\mu\in{\frak M}_x({\cal U})$, $\lambda\in{\frak M}_y({\cal U})$ such that $\mu$ is dependent and $\lambda$ is finitely satisfiable in a small set $A$.
	Suppose that $(\bar a_i)_i$ is a net  in $(A^y)^{<\omega}$ such that $Av(\bar a_i)\to \lambda$, then: $$Av(\bar a_i)\otimes\mu\to \lambda\otimes\mu.$$
\end{Lemma}
\begin{proof}
	Let $\phi(x,y)$ be a formula, and set $\lambda_i=Av(\bar a_i)$. By Fact~\ref{fundamental}, as $\mu$ is dependent, the convex hull of the set $\{f_{tp(a)}^{\phi^*}:a\in A\}$ is $\mu$-stable. (Notice that this convex hull contains each $f_{\lambda_i}^{\phi^*}$.) 
	Therefore,   $$\lim_i [\lambda_i\otimes\mu](\phi(x,y))=\lim_i\int f_{\lambda_i}^{\phi^*} d\mu \stackrel{(*)}{=} \int(\lim_i f_{\lambda_i}^{\phi^*})d\mu=\int f_\lambda^{\phi^*} d\mu=\lambda\otimes\mu.$$  The second equality $\stackrel{(*)}{=}$ holds by Fact~\ref{convergence in measure}. 
\end{proof}

The following is a generalization of \cite[Lemma~3.1]{HPS} and \cite[Proposition~5.1]{CGH}.

\begin{Theorem}  \label{commute}  
	Let $\mu_x\in{\frak M}({\cal U})$ be a    dependent measure and definable over $A$, and $\lambda_y\in{\frak M}({\cal U})$ be  $A$-finitely satisfied. Then   $\mu_x\otimes\lambda_y=\lambda_y\otimes\mu_x$.	
\end{Theorem}
\begin{proof}	
Fix $\phi(x,y;z)\in L$ and $d\in\cal U$ with $|d|=|z|$. 
 By Fact~\ref{Dirac}(iv), there is $(\bar a_i)\in (A^y)^{<\omega}$  such that   $Av(\bar a_i)\to \lambda$. Set $\lambda_i=Av(\bar a_i)$.
	\begin{align*} 	\lambda\otimes\mu(\phi(x,y;d)) &  \stackrel{(1)}{=} [\lim_i (\lambda_i\otimes\mu)]\phi(x,y;d) \\
	&  \stackrel{(2)}{=}   [\lim_i (\mu\otimes \lambda_i)]\phi(x,y;d)  \\
	&  \stackrel{(3)}{=}  [ \mu\otimes (\lim_i\lambda_i)]\phi(x,y;d)   =   \mu\otimes \lambda(\phi(x,y;d)).
	\end{align*}
(2): As the $\lambda_i$'s are trivial, they commute with another measure. (This can be easily seen by looking at the definition of the Morley product.) 
	(1) and (3) follows from Lemma~\ref{convergence 2} and  definability of $\mu$, respectively.  (Cf. also \cite[Lemma~5.4]{CGH} for a proof of (3).)
\end{proof}

In the following we give  another 
  example of non-dependent measure.
\begin{Example}
 \label{non-example}
  In Proposition~7.14 of \cite{CGH}, it is shown that there is a complete theory $T$, a 	global definable  measure $\mu_x$ and a finitely satisfiable (and definable)  type $q_y$ such that $\mu\otimes q\neq q\otimes\mu$.
 Therefore, by Theorem~\ref{commute}, $\mu$ is {\bf not} dependent. Moreover, as $q$ is dependent, definable and finitely satisfiable, by  Theorem~\ref{commute}, $\mu$ is {\bf not}  finitely satisfiable. 
 However, in the discussion after Proposition~7.14 in \cite{CGH}, a direct explanation that $\mu$ is not finitely satisfiable is presented.
 
\noindent
Also, it is shown in \cite{CGH} that the measure $\lambda^*$ in Corollary~7.15 there
 has no global extensions.
Alternatively, for the same reason,  the measure $\lambda^*$ 
 is not dependent, and so by Theorem~\ref{fim is dependent}  below $\lambda^*$ has no  $fim$ global extension. A question arises: Does there exist a dependent measure that lacks any {\em smooth} extensions (including $fim$ or definable extensions)?
\end{Example}

\begin{Definition}
{\em Let $A$ be a small set and  $\mu\in{\frak M}({\cal U})$. 
	\newline (i) We say that  $\mu$ is  {\em dfs}  over $A$ if it is both definable over and finitely satisfiable in $A$.
	\newline (ii) We say that $\mu$ is $ddfs$ (over $A$) if it is both  $dfs$ (over $A$) and dependent.}
\end{Definition}

\begin{Corollary} \label{symmetry}
Let  $\mu\in{\frak M}({\cal U})$  be $ddfs$. Then $\mu$ is symmetric, that is, for any $n\in\Bbb N$ and any permutation $\sigma$ of $\{1,\ldots,n\}$, $\mu_{x_1}\otimes\cdots\otimes\mu_{x_n}=\mu_{\sigma x_1}\otimes\cdots\otimes\mu_{\sigma x_n}$. 
\end{Corollary}
\begin{proof}
	This follows from Theorem~\ref{commute} and associativity of  $\otimes$ for definable measures.  (See \cite[Proposition~2.6]{CG} for a proof of associativity of  $\otimes$ for definable measures. Notice that the iterated products are well-defined by definability of $\mu$, $\mu\otimes\mu$ and so on.) 
\end{proof}

\begin{Remark}
 \label{optimal}
	An obvious question is whether the argument of Theorem~\ref{commute} works with a weaker condition than dependence of measures. The answer is negative from one perspective: As measurability is necessary for the definition of $\mu\otimes\lambda$ (cf. Proposition~\ref{measurable} and Definition~\ref{Morley product}), such a condition must require measurability. On the other hand, there is a weaker notion of `dependent measure' which is equivalent to  measurability, namely $R$-stable (cf. the condition (a) in Theorem~9-4-2 in \cite{T84} and \cite[465S]{Fremlin4}).\footnote{When we talk about measurability, we mean the condition (a) in Theorem~9-4-2 in \cite{T84} or Fact~\ref{closure} above.}  The only difference is in the definition of  product of measures. Therefore, the above arguments work if and only if we use the notion $R$-stable (or $R$-dependent measure) instead of dependent measure.
Therefore, $\lambda$ is $R$-stable if and only if for any finitely satisfiable measure $\mu$, the product measure $\mu\otimes\lambda$ is well-defined. Moreover, in this case, the proof of Theorem~\ref{commute} works well.
In other words, this may convince the reader that this concept (i.e.  $R$-dependent measure) is, in a way, optimal for the purposes of this paper.
\end{Remark}




\section{Dependence and $fim$}
In this section, we study the relationship between the concepts of dependence and $fim$.
Recall from \cite{HPS} that: a global measure $\mu$ is $fim$ (over a small set $A$) if
(i) for every $\phi(x,y)\in L$, and $\epsilon>0$, for   sufficiently large $n$, there is an  $L(A)$-formula $\theta_\epsilon(x_1,\ldots,x_n)$ such that $\mu^{(n)}(\theta_\epsilon)\geq 1-\epsilon$, and (ii)  for all $b$,  $|\mu(\phi(x,b))-Av(a_1,\ldots,a_n; \phi(x,b))|\leq\epsilon$	for all  $(a_1,\ldots,a_n)\in \theta_\epsilon({\cal U})$.

 To begin, we first introduce a local notion of $NIP$ and present a result related to this concept (i.e. Proposition~\ref{NIP+fs}). Then, in Proposition~\ref{fim is dependent}, we examine the relationship between the two concepts dependence and $fim$.

  \begin{Definition}
Let $A$ be a small set and $\phi(x,y)$ a formula. We say that  $\phi(x,y)$ is {\em uniformly $NIP$ in  $A$} if there is a natural number $n=n_{\phi,A}$ such that there is {\bf no} $a_1,\ldots,a_n\in A$ such that for any $I\subseteq\{1,\ldots,n\}$, ${\cal U}\models \exists y\bigwedge_{i\in I}\phi(a_i,y)\wedge\bigwedge_{i\notin I}\neg\phi(a_i,y)$.\footnote{In the notion, `uniformly' emphasized that, in contrast to `$NIP$ in a model' in \cite{KP}, there is a natural number $n_{\phi,M}$ for any formula $\phi$.}

\noindent
We say that $A$ is {\em uniformly $NIP$} if every formula is uniformly $NIP$ in $A$.
  \end{Definition} 
\begin{Remark}
	Notice that if $A=M$ is a model, then `uniformly $NIP$ in $M$' is equivalent to $NIP$ for the theory.  (To see that this condition for a model implies $NIP$,
	note that for a fixed formula $\phi(x, y)$ and a fixed $n$, there is a sentence $\psi$
	that holds in $M$ if and only if $\phi(x, y)$ is uniformly $NIP$ with bound $n$. This
	means that if every formula is uniformly $NIP$ in $M$, then every formula is
	uniformly $NIP$ in every model of the theory, and therefore the theory is $NIP$.  Recall that theory $T$ is called $NIP$ if every formula has $NIP$ for $T$.) Of course, if $A$ is not a model, this condition is strictly weaker than  $NIP$ for the theory. 
\end{Remark}


\begin{Proposition} \label{NIP+fs}
Let $A$ be a uniformly $NIP$ set and $\mu$ a global measure which is finitely satisfiable in $A$. Suppose that for each $k$, the Morley product $\mu^{(k)}$ is well-defined.Then $\mu$ is dependent.
\end{Proposition}
\begin{proof}
Suppose for a contradiction that $\mu$ is not dependent. Then, there are a measurable set $E$ with $\mu(E)>0$, and formula $\phi(x,y)$ such that for each $k$, $\mu^k(D_k({\cal U},{\cal U},E,\phi))=(\mu(E))^k>0$. (Cf. Definition~\ref{dependent measure}.) This means that for each $k$, $$\mu^k\big(E^k\cap\{\bar p:\forall I\subseteq k\exists b_I\in{\cal U}\bigwedge_{i\in I}\phi(p_i,b_I)=0\wedge\bigwedge_{i\notin}\phi(p_i,b_I)=1\}\big)>0. \ \ (*)$$
On the other hand, as $\mu$ is finitely satisfiable in $A$, the Morley product $\mu^{(k)}$ is so. Consider the following formula: $$\Phi(x_1,\ldots,x_k):=\forall  I\subseteq k\exists y_I\big(\bigwedge_{i\in  I}\neg\phi(x_i,y_I)\wedge\bigwedge_{i\notin  I}\phi(x_i,y_I)\big).$$
Notice that, by $(*)$, $\mu^{(k)}(\Phi)\geq\mu^k(D_k({\cal U},{\cal U}, E, \phi))>0$. (See also Proposition~3.3  in \cite{GH} or the explanation in Example \ref{non-example 2} above.) As $\mu^{(k)}(\Phi)>0$ and $\mu^{(k)}$ is finitely satisfiable in $A$, there are $a_1,\ldots,a_k\in A$ such that $\models\Phi(a_1,\ldots,a_k)$. As $k$ is arbitrary and $A$ is uniformly $NIP$, this is a contradiction.
\end{proof}

\medskip
 An  obvious  question is whether each $fim$ measure is dependent. A positive answer indicates that the notion `dependent measure' is necessary.  
\begin{Proposition} \label{fim is dependent}
  Every $fim$ measure is dependent.
\end{Proposition} 
\begin{proof}
	 	Let $\mu_x\in{\frak M}({\cal U})$ be $fim$ over a small set/model $A$. For any formula $\phi(x,y)$, there are formulas  $\theta_{\epsilon_n}(x_1,\ldots,x_n)\in L(A)$ such that $\mu^{(n)}(\theta_{\epsilon_n})\to 1$ as ${\epsilon_n}\to 0$, and  for all $b$,  $|\mu(\phi(x,b))-\frac{1}{n}\sum_1^n\phi(a_i,b)|\leq {\epsilon_n}$	for all  $(a_1,\ldots,a_n)\in \theta_{\epsilon_n}({\cal U})$.
Set $X_n:=\{(p_1,\ldots,p_n)\in (S_x({\cal U}))^n: p|_{x_i}=p_i (i\leq n) \text{ for some }  p\in S_{x_1,\ldots,x_n}({\cal U}) \text{ such that }\theta_{\epsilon_n}\in p\}$.

Similar to the argument of Example \ref{non-example 2}, by Proposition 3.3 of \cite{GH}, the restriction map $r:S_{x_1,\ldots,x_n}({\cal U})\to S_{x_1}({\cal U})^n$ via $p\mapsto (p_1,\ldots,p_n)$ where 
$p|_{x_i}=p_i$   ($i\leq n$), is a quotient map, and the {\em pushforward} of $\mu^{(n)}$ is $\mu^n$. 
 Therefore, $\mu^n(X_n)=
 \mu^{(n)}(r^{-1}(X_n))\geq \mu^{(n)}(\theta_{\epsilon_n})$.
 Since $\mu^{(n)}(\theta_{\epsilon_n})\to 1$ as ${\epsilon_n}\to 0$, $\mu^n(X_n)\to 1$.

 Set $Y_n=\{(p_1,\ldots,p_n)\in (S_{x_1}({\cal U}))^n \ : \ \sup_{b\in{\cal U}}\big|\mu(\phi(x,b))-\frac{1}{n}\sum_1^n\phi(p_i,b)\big|\leq\epsilon_n\}$. Clearly $X_n\subseteq Y_n$. Since $\mu^n(X_n)\to 1$ as $n\to\infty$, $\mu^n(Y_n)\to 1$.

 This is enough, by the equivalence (i)~$\iff$~(iii) of  Fact~\ref{fundamental}. (Recall also Definition~\ref{dependent measure} and Remark~\ref{remark-1}(i).)
\end{proof}

From one perspective, the following result  complements  Theorem~5.16 of \cite{CGH}. It  uses Proposition~\ref{fim is dependent} to ensure that  products of $fim$ and finitely satisfied measures make  sense.
\begin{Theorem} \label{CGH-like}
	Let $\mu_x\in{\frak M}({\cal U})$ be $fim$ (over $A$), and $\lambda_y$ be $A$-finitely satisfied. Then $\mu$ commutes with $\lambda$. 
\end{Theorem}
\begin{proof} This follows from Theorem~\ref{commute} and Proposition~\ref{fim is dependent}.
\end{proof}

One can give a proof similar to \cite[Proposition~5.15]{CGH}. The point here is that, as $\lambda$ is $A$-finitely satisfied and $\mu$ is dependent (by Proposition~\ref{fim is dependent}), every fiber function $f_{\lambda,N}^{\phi^*}$  is $\mu|_N$-measurable. (Cf. Definition~\ref{Morley product}.) We just have to check everything still works well.

\subsection*{Concluding remarks/questions}

Two key properties have allowed the generalization of model-theoretic results presented in this article. If $\lambda \in \mathfrak{M}(\mathcal{U})$ is dependent, then:
\begin{itemize}
\item[(1)] For any set $A$, the restriction $\lambda|_A$ can be approximated by types within its support (see Corollary~\ref{cor 1}).
\item[(2)] The Morley product $\mu\otimes\lambda$ is well-defined for any global measure $\mu$ that is finitely satisfiable in some small set (refer to Proposition~\ref{measurable}).
\end{itemize}

Regarding whether conditions like (1) and (2) are suitable for generalizing ``dependence" beyond the NIP framework, we offer the following observations:
First, the concept of "dependence" aligns optimally with (1) and (2), as these conditions imply dependence.\footnote{Note that, with the explanations provided in Remark \ref{optimal}, if we wish to be more precise, all the results of this paper also hold for the weaker notion of $R$-stability. However, since the concept of $\mu$-stability has been emphasized more prominently in the literature, particularly in the book by Fremlin, we have focused on this latter notion.
In summary, $\mu$-stability is equivalent to (1)+(2), while $R$-stability is equivalent to (2). In the case of R-stability, (2) implies (1).
(What we mean by (2) is something similar to the usage in Proposition \ref{measurable} and Definition \ref{Morley product}, and
what we mean by condition (1) is something similar to Fact \ref{fundamental}(iii), or Theorem \ref{almost fim} and Corollary \ref{cor 1}.)
Note that the recent points do not contradict the fact that $\mu$-stability is stronger than $R$-stability, as the definition of the product measure differs in these two cases.}
Second, almost all established model-theoretic results in the NIP setting can be generalized to arbitrary theories using (1) and (2). However, certain results, such as Theorem~\ref{commute}, Lemma~\ref{convergence 2}, and even Theorem~\ref{almost fim}, are novel and cannot be easily derived through this approach.
Finally, and most significantly, exploring connections between distinct areas, such as model theory and measure theory, is invaluable. Mathematics as a discipline thrives on such cross-disciplinary links rather than limiting itself to introspective analysis. 
\newline
On the other hand, as the notion of $\mu$-stable is defined for real-valued functions, all the results of this article can be easily generalized to continuous logic \cite{BBHU}.

\medskip
At the end paper let us ask the following questions:
\begin{Question}
	(i) Is the product of two global $ddfs$ measures   $ddfs$?
	\newline
	(ii) Is every $dfs$ measure dependent? If so, by Theorem~\ref{commute}, the answer to Question~5.10 of \cite{CGH} is positive (i.e., any two $dfs$  global measures commute.)   As the product of two $dfs$ measures is $dfs$, a positive answer to (ii) also automatically gives a positive answer to (i).
\end{Question}
We strongly believe that the answer to (i) is negative. Indeed, suppose that $f:S_\phi({\cal U})\times S_{\phi^*}({\cal U})\to [0,1]$ is a function and for all $p\in S_\phi({\cal U})$ and 
$q\in S_{\phi^*}({\cal U})$, $x$-sections $f_p:S_{\phi^*}({\cal U})\to[0,1]$ and $y$-sections $f^q:S_{\phi}({\cal U})\to[0,1]$ are measurable. There is no guarantee that $f$ will be measurable.
A similar idea may  lead to the rejection of a claim in the initial version of \cite{CGH} that $fim$ measures are closed under Morley product, i.e. the products of  $fim$ measures are $fim$.\footnote{We clarify that the published version of \cite{CGH} does not claim that the product of $fim$ measures is $fim$. } 
  Finally, we believe that the answer to Question~5.10 of \cite{CGH} is negative, however, we should wait for such counterexamples in future work.  In $NIP$ theories, the answer to these questions is clearly positive. So we can make the questions more accurate.
\begin{Question}
	(i) In which theories is the product of two global $ddfs$ measures    $ddfs$? 
	(ii) In which theories is  every $dfs$ measure  dependent? 
\end{Question}


\begin{Question} Is the product of dependent measures (assuming the product is well-defined) dependent?
\end{Question}

\bigskip\noindent
{\bf Acknowledgements.}
 I want to thank Gabriel Conant
for his interest in reading a preliminary version of this article and for his comments. I thank the anonymous referee for his details suggestions and corrections; they helped to improve significantly the exposition if this paper.

I would like to thank the Institute for Basic Sciences (IPM), Tehran, Iran. Research partially supported by IPM grant 1402030028.


\begin{thebibliography}{99} \label{ref}
	
	
	
	
	\bibitem[BBHU08]{BBHU} I. Ben-Yaacov, A. Berenstein, C. W. Henson, A. Usvyatsov, \emph{Model theory for metric structures}, Model theory with Applications to Algebra and Analysis, vol. 2  (Z. Chatzidakis, D. Macpherson, A. Pillay, and A. Wilkie, eds.),  London Math Society Lecture Note Series, vol. 350, Cambridge University Press, 2008.
	
	
	\bibitem[BFT78]{BFT} J. Bourgain, D. H. Fremlin,   M. Talagrand. Pointwise compact sets of baire-measurable functions. American Journal of Mathematics, 100(4):pp. 845-886, 1978.
	
	
	
	\bibitem[ChG21]{ChG} A. Chernikov, K. Gannon, Definable convolution and idempotent Keisler measures, Israel Journal of Mathematics,  {\bf 248}, pages 271–314 (2022).
	
	\bibitem[CS16]{CS} A. Chernikov, S. Stracheko, Definable regularity lemma for NIP hypergraphs, https://arxiv.org/abs/1607.07701
	
	\bibitem[CG20]{CG} G. Conant, K. Gannon, Remarks on generic stability in independent theories, Ann. Pure	Appl. Logic 171 (2020), no. 2, 102736, 20. MR 4033642
		
			
	\bibitem[CGH23]{CGH} G. Conant, K. Gannon, J. Hanson,  Keisler measures in the wild, arxiv 2023  https://arxiv.org/abs/2103.09137v3

	
	
	
	
	
	
	\bibitem[Fre06]{Fremlin4} D. H. Fremlin, {\em Measure Theory}, vol.4, (Topological Measure Spaces, Torres Fremlin, Colchester, 2006). (See  https://www1.essex.ac.uk/maths/people/fremlin/cont46.htm )
	
	
	\bibitem[G19]{Gannon} K. Gannon, Local Keisler measures and NIP formulas. The Journal of Symbolic Logic, 84(3):1279-1292, 2019
	
	\bibitem[G20]{Gannon-thesis} K. Gannon, Approximation theorems for Keisler measures, Ph.D. thesis, University of Notre Dame, 2020.
	
	
	
	\bibitem[G22]{Gannon sequential} K. Gannon, Sequential approximations for types and Keisler measures,   Fundamenta Mathematicae, 257, 305-336 (2022).
	
	\bibitem[GH24]{GH} K. Gannon, J.E. Hanson, Model theoretic events,   https://arxiv.org/abs/2402.15709
	

\bibitem[GOU13]{GOU} C. Garcia, A. Onshuus, A. Usvyatsov,  Generic stability, forking, and thorn-forking, Transactions of the American Mathematical Society, Vol. 365, No. 1, (2013)  pp. 1--22 (22 pages)


	
	
	
	
	
	\bibitem[HP11]{HP} E. Hrushovski,  A. Pillay, On $NIP$ and invariant measures, Journal of the European Mathematical Society, 13 (2011), 1005-1061.
	
	\bibitem[HPS13]{HPS} E. Hrushovski, A.  Pillay, and P. Simon. Generically stable and smooth measures in NIP	theories. Transactions of the American Mathematical Society 365.5 (2013): 2341-2366.
	
	\bibitem[Iba14]{Iba14} T. Ibarluc\'{i}a, The dynamical hierachy for Roelcke precompact Polish groups, arXiv:1405.4613v1, 2014.
	
	
	
	
	
	
	
	
	\bibitem[Kha16]{K-amenable}	K. Khanaki, Amenability, extreme amenability, model-theoretic stability, and NIP in integral 	logic, Fundamenta Mathematicae 234 (2016), 253-286
	
	\bibitem[Kha17]{K-definable} K. Khanaki,  NIP formulas and Baire 1 definability,   unpublished research note, arXiv:1703.08731, 2017.
	
	\bibitem[Kha20]{K3} K. Khanaki,  Stability, the NIP, and the NSOP; Model Theoretic Properties of Formulas via Topological Properties of Function Spaces,  Math. Log. Quart. 66, No. 2, 136-149  (2020) / DOI 10.1002/malq.201500059
	
	
	\bibitem[Kha21]{K-generic} K. Khanaki,  Generic stability and modes of convergence,  submitted, https://arxiv.org/abs/2204.03910, 2021.
	
	
	\bibitem[Kha23]{K-Morley} K. Khanaki,  Remarks on convergence of Morley sequences, JSL, (2023) DOI: https://doi.org/10.1017/jsl.2023.18
	
	
	\bibitem[Kha24]{K-GC} K. Khanaki,  Glivenko-Cantelli classes and $NIP$ formulas,   Archive for Mathematical Logic,   Volume 63, pages 1005–1031, (2024)
	
	
	
	
	
	
	
	
	

	
	\bibitem[KP18]{KP} K. Khanaki, A. Pillay, Remarks on $NIP$ in a model, Math. Log. Quart. 64, No. 6, 429-434 (2018) / DOI 10.1002/malq.201700070
	
	
	
	\bibitem[P18]{Pillay} A. Pillay, Generic stability and Grothendieck, South American Journal 	of Logic Vol. 2, n. 2,(2016), p. 1-6.
	
	\bibitem[PT11]{Pillay-Tanovic} A. Pillay,  P.  Tanović, Generic stability, regularity, and quasiminimality, Models, logics, and higher-dimensional categories, CRM Proc. Lecture Notes, vol. 53, Amer. Math. Soc., Providence, RI, 2011, pp. 189–211. MR 2867971
	
	
	
	
	
	\bibitem[S15]{Simon} P. Simon. A guide to NIP theories. Cambridge University Press, 2015.
	
	
	\bibitem[S15a]{Simon-Rosenthal} P. Simon, Rosenthal compacta and NIP formulas, Fundamenta Mathematicae, 231, n. 1, 81-91, 2015
	
	\bibitem[T84]{T84} M. Talagrand, Pettis integral and measure theory. Mem. Amer. Math. Soc. 307 (1984).
	
	
	\bibitem[T87]{Talagrand} M. Talagrand, The  Glivenko--Cantelli problem, Ann. Probab.  15  (1987),  837-870.
	
	
	
\end{thebibliography}
\end{document}